\documentclass[11pt]{article}
\usepackage[T1]{fontenc}
\usepackage[margin=1in]{geometry}
\usepackage[colorlinks=true,linkcolor=blue,citecolor=blue,urlcolor=blue]{hyperref}

\usepackage{amssymb}
\usepackage{amsmath}
\usepackage{amsthm}

\newtheorem{theorem}{Theorem}
\newtheorem{corollary}[theorem]{Corollary}
\newtheorem{definition}[theorem]{Definition}
\begin{document}

\title{A finite-sample Borel--Cantelli inequality under $m$-dependence}
\author{Chatchawan Panraksa\\
Applied Mathematics Program, Mahidol University International College\\
999 Phutthamonthon 4 Road, Salaya, Nakhonpathom 73170, Thailand\\
\texttt{chatchawan.pan@mahidol.ac.th}}
\date{}

\maketitle

%% Abstract
\begin{abstract}
We prove an explicit finite-sample version of the Borel--Cantelli
lemma under $m$-dependence.  Given any $m$-dependent sequence of
events $(A_k)_{1\leq k\leq N}$, we show that
\[
  \mathbb{P}\Bigl(\bigcup_{k=1}^N A_k\Bigr)
  \ge 1 - \exp\Bigl(-\frac{1}{m+1}
    \sum_{k=1}^{N} \mathbb{P}(A_k)\Bigr).
\]
The proof splits the index set into residue classes modulo $m+1$, so
that each class consists of mutually independent events, and then
applies an elementary product--to--exponential bound.  We further
derive a quantitative windowed corollary: if the partial sums satisfy
\(\sum_{k=1}^{\phi(n)}\mathbb{P}(A_k)\ge n\) for all \(n\ge1\), then for
every \(N\ge1\) and \(i\ge0\),
\[
  \mathbb{P}\Bigl(\bigcup_{k=i+1}^{\phi(i+N)} A_k\Bigr)
  \ge 1-\exp\Bigl(-\frac{N}{m+1}\Bigr).
\]
Finally, we present a complementary second-order refinement involving
local pairwise intersection probabilities.  These results complement
the asymptotic and rate results of Lu, Shi and Zhao (2026) by
providing explicit finite-$N$ bounds and a simple comparison framework
for the baseline and second-order estimates.
\end{abstract}

\medskip
\noindent\textbf{Keywords:}\\
Borel--Cantelli lemma; $m$-dependence; quantitative probability;
non-asymptotic bounds

\smallskip
\noindent\textbf{MSC 2020:} 60F15, 60F20

\section{Introduction}
The Borel--Cantelli lemma is a cornerstone of probability theory.  In
its classical form it asserts that, for a sequence of events
\((A_k)_{k\geq 1}\) in a probability space, if
\(\sum_k \mathbb{P}(A_k)<\infty\) then only finitely many of the
events occur almost surely, whereas if the events are mutually
independent and the series diverges then infinitely many of them
occur almost surely.  See, for example, Durrett~\cite{Durrett2010}
for a textbook treatment.  Since independence is often an idealised
assumption, a substantial body of work has explored relaxations of
independence.  Notable examples include the sharp criteria of
Erd\H{o}s and R\'enyi~\cite{ErdosRenyi1959}, Ortega and
Wschebor~\cite{OrtegaWschebor1984}, and the quantitative bounds of
Chandra~\cite{Chandra1999} and Petrov~\cite{Petrov2002}; see
Arthan and Oliva~\cite{ArthanOliva2021} for a recent survey.

Recently Lu, Shi and Zhao~\cite{LuShiZhao2026} studied the second
Borel--Cantelli lemma under $m$-dependence.  A sequence of events
\((A_k)_{k\geq 1}\) is called \emph{$m$-dependent} if any two
subcollections of events separated by more than $m$ indices are
independent.  Their main result (Theorem~4 in
\cite{LuShiZhao2026}) shows that if \(\sum_{k=1}^\infty
\mathbb{P}(A_k)=\infty\) then \(\limsup_{k\to\infty} A_k\) has
probability one; Theorem~7 of the same paper gives a quantitative rate
version when the partial sums grow linearly, and the block/parity
reduction in that setting leads to the coefficient \(1/(2m)\), which
we mention here for later comparison.  These results are
asymptotic in $N$ and concern almost-sure eventual occurrence, whereas
the present note gives explicit lower bounds for finite unions.  In
particular, Theorem~\ref{thm:finite} produces a concrete coefficient
\(1/(m+1)\) for
\(\mathbb{P}(\bigcup_{k=1}^N A_k)\), and Corollary~\ref{cor:window}
turns the same estimate into a finite-window statement.

Classical references such as Erd\H{o}s--R\'enyi~\cite{ErdosRenyi1959},
Ortega--Wschebor~\cite{OrtegaWschebor1984}, Chandra~\cite{Chandra1999},
and Petrov~\cite{Petrov2002} provide important context for refinements
of Borel--Cantelli-type arguments, but the emphasis there is different
from the one adopted here.  The novelty of the present note lies in
recording an explicit finite-$N$ lower bound under $m$-dependence with
a simple residue-class proof, together with a windowed corollary and a
second-order local-overlap variant that can be compared directly with
the baseline estimate.

The aim of this note is to provide simple finite-sample analogues with
explicit constants.  Our first theorem gives a non-asymptotic lower
bound on the probability that at least one of $N$ $m$-dependent
events occurs, namely
\[
  \mathbb{P}\Bigl(\bigcup_{k=1}^N A_k\Bigr)
  \ge 1-\exp\Bigl(-\frac{1}{m+1}\sum_{k=1}^N \mathbb{P}(A_k)\Bigr).
\]
The key observation is that the index set \(\{1,\dots,N\}\) can be
split into \(m+1\) residue classes modulo \(m+1\), and the events in
each class are mutually independent.  Applying the elementary bound
\(\prod_j(1-x_j)\le \exp(-\sum_j x_j)\) within each class and then
selecting a class with maximal total mass yields the stated estimate.

In addition to this baseline bound, we present a complementary
second-order inequality obtained from Bonferroni estimates on shifted
block partitions.  The resulting exponent is expressed in terms of
local intersection probabilities \(\mathbb{P}(A_i\cap A_j)\) for
pairs with \(|i-j|\le m-1\).  The comparison carried out after
Theorem~\ref{thm:cov} shows exactly when this second-order bound
improves on the baseline estimate.  Finally, we deduce a finite-window
corollary that
parallels the rate result of Lu, Shi and Zhao~\cite{LuShiZhao2026}
while retaining explicit constants.

\section{Preliminaries}
We recall some basic definitions and a simple exponential bound.

\begin{definition}[\boldmath$m$-dependence]
\label{def:mdep}
Let $(A_k)_{k\ge1}$ be a sequence of events on a probability space.
For finite index sets $I,J\subseteq\mathbb{N}$, denote the distance
between $I$ and $J$ by \(\mathrm{dist}(I,J)=\inf\{|i-j|:i\in I,
j\in J\}.\)  The sequence is called $m$-dependent if whenever
\(\mathrm{dist}(I,J)>m\) the sigma-algebras generated by
\(\{A_i:i\in I\}\) and \(\{A_j:j\in J\}\) are independent.  In
particular, events $A_i$ and $A_j$ are independent whenever
\(|i-j|>m\).
\end{definition}

An immediate consequence is that if \(J\subseteq\mathbb{N}\) satisfies
\(|i-j|>m\) for all distinct \(i,j\in J\), then the family
\(\{A_k:k\in J\}\) is mutually independent.  Indeed, one may apply
Definition~\ref{def:mdep} inductively to the sigma-algebra generated by
the previously selected indices.  In particular, for each residue class
\(r\in\{1,\dots,m+1\}\), the events with indices congruent to
\(r\pmod{m+1}\) are mutually independent.

\medskip
\noindent\textbf{Product-to-exponential bound.}
For any integer \(L\ge1\) and any numbers \(0\le x_1,\dots,x_L\le1\),
\[
  \prod_{j=1}^L (1 - x_j) \leq \exp\bigl(-\sum_{j=1}^L x_j\bigr).
\]
This follows immediately from the elementary inequality
\(\log(1-x)\le -x\) for \(x\in[0,1]\).

We shall use this standard bound repeatedly to convert finite products
of probabilities into exponentials of sums.

\section{Main results}

Our first theorem is a finite-sample version of the second
Borel--Cantelli lemma under $m$-dependence.

\begin{theorem}[Finite-sample $m$-dependence inequality]
\label{thm:finite}
Let $m\ge 1$ be an integer and let $(A_k)_{1\le k\le N}$ be a finite
sequence of $m$-dependent events.  Write
\[
  S_N=\sum_{k=1}^{N} \mathbb{P}(A_k).
\]
Then
\[
  \mathbb{P}\Bigl(\bigcup_{k=1}^N A_k\Bigr)
  \ge 1 - \exp\Bigl(- \frac{1}{m+1}\,S_N\Bigr).
\]
\end{theorem}

\begin{proof}
For each residue class $r\in\{1,\dots,m+1\}$, define
\[
  J_r=\{\,k\in\{1,\dots,N\}: k\equiv r \pmod{m+1}\,\}.
\]
If $i,j\in J_r$ with $i\ne j$, then $|i-j|$ is a nonzero multiple of
$m+1$, hence $|i-j|\ge m+1>m$.  Therefore the events
$\{A_k:k\in J_r\}$ are mutually independent for each fixed $r$.

By De Morgan's law,
\[
  \mathbb{P}\Bigl(\bigcup_{k=1}^N A_k\Bigr)
  =1-\mathbb{P}\Bigl(\bigcap_{k=1}^N A_k^{\mathrm c}\Bigr)
  =1-\mathbb{P}\Bigl(\bigcap_{r=1}^{m+1}\ \bigcap_{k\in J_r} A_k^{\mathrm c}\Bigr).
\]
Hence
\[
  \mathbb{P}\Bigl(\bigcap_{k=1}^N A_k^{\mathrm c}\Bigr)
  \le
  \min_{1\le r\le m+1}
  \mathbb{P}\Bigl(\bigcap_{k\in J_r} A_k^{\mathrm c}\Bigr).
\]
Since the events within each $J_r$ are mutually independent, we obtain
\[
  \mathbb{P}\Bigl(\bigcap_{k=1}^N A_k^{\mathrm c}\Bigr)
  \le
  \min_{1\le r\le m+1}
  \prod_{k\in J_r}\bigl(1-\mathbb{P}(A_k)\bigr).
\]
Using $1-x\le e^{-x}$ for $x\in[0,1]$, this gives
\[
  \mathbb{P}\Bigl(\bigcap_{k=1}^N A_k^{\mathrm c}\Bigr)
  \le
  \min_{1\le r\le m+1}
  \exp\Bigl(-\sum_{k\in J_r}\mathbb{P}(A_k)\Bigr).
\]
Because the sets $J_1,\dots,J_{m+1}$ form a partition of
$\{1,\dots,N\}$,
\[
  \sum_{r=1}^{m+1}\sum_{k\in J_r}\mathbb{P}(A_k)=S_N.
\]
Therefore at least one residue class satisfies
\[
  \sum_{k\in J_r}\mathbb{P}(A_k)\ge \frac{S_N}{m+1},
\]
and so
\[
  \mathbb{P}\Bigl(\bigcap_{k=1}^N A_k^{\mathrm c}\Bigr)
  \le \exp\Bigl(-\frac{1}{m+1}S_N\Bigr).
\]
Consequently,
\[
  \mathbb{P}\Bigl(\bigcup_{k=1}^N A_k\Bigr)
  \ge 1-\exp\Bigl(-\frac{1}{m+1}S_N\Bigr),
\]
as claimed.  Empty products are understood as $1$.
\end{proof}

The factor $1/(m+1)$ comes from splitting the index set into $m+1$
residue classes modulo $m+1$.  For $m=1$ this coincides with the
coefficient $1/(2m)$ obtained from the block/parity argument, while for
$m\ge2$ it is strictly sharper.  We state Theorem~\ref{thm:finite} for
$m\ge1$; the independent case $m=0$ is classical and yields
\[
  \mathbb{P}\Bigl(\bigcup_{k=1}^N A_k\Bigr)\ge 1-\exp(-S_N).
\]

\begin{corollary}[Finite-window rate form]
\label{cor:window}
Let $(A_k)_{k\geq1}$ be an $m$-dependent sequence of events and let
\(\phi:\mathbb{N}\to\mathbb{N}\) be a non-decreasing function with
\[
  \sum_{k=1}^{\phi(n)} \mathbb{P}(A_k)\ge n
  \qquad\text{for all } n\ge 1.
\]
Then for every integer $N\ge1$ and all $i\ge0$,
\[
  \mathbb{P}\Bigl( \bigcup_{k=i+1}^{\phi(i+N)} A_k \Bigr)
  \ge 1 - \exp\Bigl(-\frac{N}{m+1}\Bigr).
\]
\end{corollary}

\begin{proof}
Fix $i\ge0$ and $N\ge1$, and set
\[
  I_{i,N}=\{i+1,i+2,\dots,\phi(i+N)\}.
\]
Then
\[
  \sum_{k\in I_{i,N}} \mathbb{P}(A_k)
  =
  \sum_{k=1}^{\phi(i+N)} \mathbb{P}(A_k)
  -
  \sum_{k=1}^{i} \mathbb{P}(A_k).
\]
By assumption,
\[
  \sum_{k=1}^{\phi(i+N)} \mathbb{P}(A_k)\ge i+N.
\]
Also, since \(\mathbb{P}(A_k)\le 1\) for every \(k\),
\[
  \sum_{k=1}^{i} \mathbb{P}(A_k)\le i.
\]
Therefore
\[
  \sum_{k\in I_{i,N}} \mathbb{P}(A_k)\ge N.
\]
The family \(\{A_k:k\in I_{i,N}\}\) remains $m$-dependent.  Applying
Theorem~\ref{thm:finite} to this family, we obtain
\[
  \mathbb{P}\Bigl(\bigcup_{k\in I_{i,N}} A_k\Bigr)
  \ge 1-\exp\Bigl(-\frac{N}{m+1}\Bigr),
\]
which is exactly the stated bound.
\end{proof}

The bound in Corollary~\ref{cor:window} mirrors the rate result in
Theorem~7 of Lu, Shi and Zhao~\cite{LuShiZhao2026}, but now the
constant \(1/(m+1)\) and the window length are explicit.

\subsection{Second-order correction via shifted blocks}
Although Theorem~\ref{thm:finite} follows from a direct splitting of the
index set into $m+1$ independent residue classes, it is sometimes useful to
have an alternative bound that depends explicitly on \emph{local} pairwise
overlaps.  The next result is a second-order (Bonferroni-type) inequality
obtained by averaging over shifted block partitions.  Depending on the size
of the local intersection terms, the resulting exponent may be sharper or
weaker than that of Theorem~\ref{thm:finite}.

\begin{theorem}[Local-intersection inequality]\label{thm:cov}
Under the assumptions of Theorem~\ref{thm:finite}, we have
\begin{equation}\label{eq:cov-refined}
  \mathbb{P}\Bigl(\bigcup_{k=1}^N A_k\Bigr)
  \;\ge\;
  1-\exp\Biggl(
   -\frac{1}{2}\sum_{k=1}^N \mathbb{P}(A_k)
   \;+\;\frac{1}{2}\!\!\sum_{\substack{1\le i<j\le N\\|i-j|\le m-1}}
   \mathbb{P}(A_i\cap A_j)
  \Biggr).
\end{equation}
\end{theorem}

\begin{proof}
\textit{Step 1 (Shifted block partitions and $1$-dependence).}\par
For each $r\in\{0,1,\dots,m-1\}$, define the $r$-shifted partition of
$\{1,\dots,N\}$ into disjoint length-$m$ blocks by
\[
\begin{aligned}
  I^{(r)}_{j}
  ={}&\{\,r+(j-1)m+1,\ r+(j-1)m+2,\ \dots,\ r+jm\,\}\\
      &{}\cap\{1,\dots,N\}, \qquad j=0,1,2,\dots,
\end{aligned}
\]
and the corresponding block events
\[
  B^{(r)}_{j}=\bigcup_{k\in I^{(r)}_{j}} A_k.
\]
For $j=0$ the block may be shorter; the sets $(I^{(r)}_{j})_{j\ge0}$ are
disjoint and their union is $\{1,\dots,N\}$.  Moreover, the block sequence
$\bigl(B^{(r)}_{j}\bigr)_{j\ge0}$ is $1$-dependent: if $|j-j'|\ge2$ and
$k\in I^{(r)}_{j}$, $k'\in I^{(r)}_{j'}$, then $|k-k'|\ge m+1>m$, so
$B^{(r)}_{j}$ and $B^{(r)}_{j'}$ are independent by $m$-dependence.

\medskip
\textit{Step 2 (Parity splitting and product-to-exponential).}
Fix $r$.  Let $\mathcal{J}_{\mathrm{odd}}^{(r)}$ and
$\mathcal{J}_{\mathrm{even}}^{(r)}$ be the odd/even block indices.  Since the
blocks within each parity are mutually independent and
\[
  \bigcup_{k=1}^N A_k = \bigcup_{j\ge0} B^{(r)}_{j},
\]
we have
\begin{align*}
\mathbb{P}\Bigl(\bigcap_{k=1}^N A_k^{\mathrm c}\Bigr)
  &= \mathbb{P}\Bigl(\bigcap_{j\ge0}(B^{(r)}_{j})^{\mathrm c}\Bigr)\\
  &\le \min\!\Bigl\{
  \prod_{j\in\mathcal{J}_{\mathrm{odd}}^{(r)}}\bigl(1-\mathbb{P}(B^{(r)}_{j})\bigr),\,
  \prod_{j\in\mathcal{J}_{\mathrm{even}}^{(r)}}\bigl(1-\mathbb{P}(B^{(r)}_{j})\bigr)
  \Bigr\}.
\end{align*}
Using $1-x\le e^{-x}$ and $\min\{e^{-x},e^{-y}\}\le e^{-(x+y)/2}$, we obtain
\begin{equation}\label{eq:parity-exp}
  \mathbb{P}\Bigl(\bigcup_{k=1}^N A_k\Bigr)
  \;\ge\; 1-\exp\Bigl(-\frac{1}{2}\sum_{j\ge0}\mathbb{P}(B^{(r)}_{j})\Bigr)
  \qquad\text{for each }r\in\{0,1,\dots,m-1\}.
\end{equation}

\medskip
\textit{Step 3 (Second-order Bonferroni inside blocks).}
Bonferroni’s second-order inequality yields
\[
  \mathbb{P}\bigl(B^{(r)}_{j}\bigr)
  \ge
  \sum_{i\in I^{(r)}_{j}} \mathbb{P}(A_i)
  -
  \sum_{\substack{i<\ell\\ i,\ell\in I^{(r)}_{j}}} \mathbb{P}(A_i\cap A_\ell).
\]
Summing over $j\ge0$ and averaging over $r=0,\dots,m-1$ gives
\begin{equation}\label{eq:block-avg}
  \frac{1}{m}\sum_{r=0}^{m-1}\ \sum_{j\ge0}\ \mathbb{P}\bigl(B^{(r)}_{j}\bigr)
  \ge
  \sum_{k=1}^N \mathbb{P}(A_k)
  -
  \frac{1}{m}\sum_{r=0}^{m-1}\ \sum_{j\ge0}
  \sum_{\substack{i<\ell\\ i,\ell\in I^{(r)}_{j}}} \mathbb{P}(A_i\cap A_\ell).
\end{equation}

\medskip
\textit{Step 4 (Pair-counting across shifted partitions).}
Fix $1\le i<\ell\le N$ with gap $d=\ell-i$.  The pair $(i,\ell)$ can lie in
a common length-$m$ block only when $d\le m-1$, and in that case this occurs
for exactly $m-d$ values of $r\in\{0,\dots,m-1\}$.  Consequently,
\[
  \frac{1}{m}\sum_{r=0}^{m-1}\ \sum_{j\ge0}
  \sum_{\substack{i<\ell\\ i,\ell\in I^{(r)}_{j}}} \mathbb{P}(A_i\cap A_\ell)
  \le
  \sum_{\substack{1\le i<\ell\le N\\ |i-\ell|\le m-1}} \mathbb{P}(A_i\cap A_\ell).
\]
Plugging this into \eqref{eq:block-avg} yields
\begin{equation}\label{eq:LB-sumB}
  \frac{1}{m}\sum_{r=0}^{m-1}\ \sum_{j\ge0}\ \mathbb{P}\bigl(B^{(r)}_{j}\bigr)
  \ge
  \sum_{k=1}^N \mathbb{P}(A_k)
  -
  \sum_{\substack{1\le i<\ell\le N\\ |i-\ell|\le m-1}} \mathbb{P}(A_i\cap A_\ell).
\end{equation}

\medskip
\textit{Step 5 (From block sums to the exponential bound).}
Let
\[
  X_r:=\sum_{j\ge0}\mathbb{P}(B^{(r)}_{j}).
\]
From \eqref{eq:parity-exp} we have
\[
  \mathbb{P}\Bigl(\bigcup_{k=1}^N A_k\Bigr)
  \ge 1-\exp\Bigl(-\frac{1}{2}X_r\Bigr)
  \qquad\text{for each }r.
\]
Hence
\[
  \mathbb{P}\Bigl(\bigcup_{k=1}^N A_k\Bigr)
  \ge 1-\exp\Bigl(-\frac{1}{2}\max_{0\le r\le m-1} X_r\Bigr)
  \ge 1-\exp\Bigl(-\frac{1}{2m}\sum_{r=0}^{m-1} X_r\Bigr).
\]
Using \eqref{eq:LB-sumB} to lower bound
\(\frac{1}{m}\sum_{r=0}^{m-1} X_r\), we obtain
\[
  \mathbb{P}\Bigl(\bigcup_{k=1}^N A_k\Bigr)
  \ge
  1-\exp\Biggl(
   -\frac{1}{2}\sum_{k=1}^N \mathbb{P}(A_k)
   +\frac{1}{2}\!\!\sum_{\substack{1\le i<\ell\le N\\|i-\ell|\le m-1}}
   \mathbb{P}(A_i\cap A_\ell)
  \Biggr),
\]
which is exactly \eqref{eq:cov-refined}.
\end{proof}

\medskip
\noindent\textbf{Comparison with Theorem~\ref{thm:finite}.}
Let
\[
  T_{m-1}:=\sum_{\substack{1\le i<j\le N\\|i-j|\le m-1}}
  \mathbb{P}(A_i\cap A_j).
\]
The exponent in Theorem~\ref{thm:cov} is \(\frac12(S_N-T_{m-1})\),
whereas Theorem~\ref{thm:finite} gives the exponent \(S_N/(m+1)\).
Hence Theorem~\ref{thm:cov} is sharper than Theorem~\ref{thm:finite}
precisely when
\[
  \frac{1}{2}(S_N-T_{m-1})>\frac{1}{m+1}S_N,
\]
that is,
\[
  T_{m-1}<\frac{m-1}{m+1}S_N.
\]
In particular, when \(m=1\) we have \(T_0=0\), so
Theorem~\ref{thm:cov} and Theorem~\ref{thm:finite} yield the same
exponent \(S_N/2\).

\section{Discussion}
Theorem~\ref{thm:finite} gives an explicit, non-asymptotic Borel--Cantelli
bound under $m$-dependence.  It complements the asymptotic second
Borel--Cantelli result (Theorem~4 in~\cite{LuShiZhao2026}) and the
quantitative rate statement (Theorem~7 in~\cite{LuShiZhao2026}) by
providing a finite-$N$ guarantee with an explicit coefficient,
depending only on the dependence range $m$ and the mass
\[
  S_N=\sum_{k=1}^N \mathbb{P}(A_k).
\]
In particular, Theorem~\ref{thm:finite} should be viewed as a
finite-union estimate rather than as a direct second
Borel--Cantelli theorem.  In contrast to classical relaxations of
independence
(Erd\H{o}s--R\'enyi~\cite{ErdosRenyi1959},
Ortega--Wschebor~\cite{OrtegaWschebor1984}, Chandra~\cite{Chandra1999},
Petrov~\cite{Petrov2002}), the inequality is fully explicit and
requires no limit transitions.

The coefficient \(1/(m+1)\) in Theorem~\ref{thm:finite} arises from the
direct decomposition of \(\{1,\dots,N\}\) into \(m+1\) residue classes
modulo \(m+1\), followed by selection of a class with maximal total
probability mass.  The second-order bound in Theorem~\ref{thm:cov}
provides a complementary inequality in which the exponent is expressed
through the local intersection probabilities \(\mathbb{P}(A_i\cap A_j)\)
for pairs with \(|i-j|\le m-1\).  The comparison above shows that this
second-order estimate is sharper exactly when
\(T_{m-1}<\frac{m-1}{m+1}S_N\), and in particular it does not improve
Theorem~\ref{thm:finite} when \(m=1\).  Finally, the windowed corollary
(Corollary~\ref{cor:window}) turns Theorem~\ref{thm:finite} into an
operational sliding-window statement that parallels the rate viewpoint
of~\cite{LuShiZhao2026} but holds uniformly for every finite~$N$.

\section*{Acknowledgments}
The author gratefully acknowledges support from Mahidol University
International College (MUIC), Thailand, and Research Contract No.~06-2025.
The author sincerely thanks the anonymous referee for a careful reading of the manuscript and for valuable comments and constructive suggestions, which helped improve the presentation and strengthen the results of this paper.


\begin{thebibliography}{00}

\bibitem{Durrett2010}
Durrett, R., 2010. Probability: Theory and Examples, fourth ed.
Cambridge University Press, Cambridge.

\bibitem{ErdosRenyi1959}
Erd\H{o}s, P., R\'enyi, A., 1959. On Cantor's series with convergent
$\sum 1/q_n$. Ann. Univ. Sci. Budapest. E\"otv\"os Sect. Math. 2~(3), 93--109.

\bibitem{OrtegaWschebor1984}
Ortega, J., Wschebor, M., 1984. On the sequence of partial maxima of some
random sequences. Stochastic Process. Appl. 16~(1), 85--98.

\bibitem{Chandra1999}
Chandra, T.K., 1999. A First Course in Asymptotic Theory of Statistics.
Narosa Pub. House, New Delhi.

\bibitem{Petrov2002}
Petrov, V.V., 2002. A note on the Borel--Cantelli lemma.
Statist. Probab. Lett. 58~(3), 283--286.

\bibitem{ArthanOliva2021}
Arthan, R., Oliva, P., 2021. On the Borel--Cantelli Lemmas, the
Erd\H{o}s--R\'enyi theorem, and the Kochen--Stone theorem.
J. Log. Anal. 13~(6), 1--23.

\bibitem{LuShiZhao2026}
Lu, D., Shi, Y., Zhao, J., 2026. The Borel--Cantelli lemma under
$m$-dependence. Statist. Probab. Lett. 227, 110525.

\end{thebibliography}
\end{document}